\documentclass[a4paper,11pt]{amsart}
\usepackage{amsfonts,amsthm,amssymb}
\usepackage{array}
\oddsidemargin=-\evensidemargin
\theoremstyle{plain}
\newtheorem{theorem}{Theorem}
\newtheorem{proposition}{Proposition}
\newtheorem{corollary}{Corollary}
\newtheorem{lemma}{Lemma}

\makeatletter
\def\udots{\mathinner{\mkern2mu\raise\p@\hbox{.}\mkern2mu\raise4\p@\hbox{.}
\mkern2mu\raise7\p@\vbox{\kern7\p@\hbox{.}} \mkern1mu}}

\newlength{\dotlength}
\makeatother
\newcommand{\uline}[1]{{\hbox to 0pt{\underline{\hphantom{{#1}}}\hss}{#1}}}

\newcommand{\angs}[1]{\left\langle #1 \right\rangle}

\newcommand\cp{\cite{Popov}}
\newcommand\gtd{\mathop{\mathrm{gtd}}}
\newcommand\rk{\mathop{\mathrm{rk}}}
\newcommand\Ker{\mathop{\mathrm{Ker}}}
\newcommand\Lie{\mathop{\mathrm{Lie}}}

\newsavebox{\sbsmallzero}
\sbox{\sbsmallzero}{$\textstyle0$}
\newsavebox{\sblargezero}
\sbox{\sblargezero}{\Huge 0}
\newlength{\ldiffzero}
\ldiffzero=\ht\sblargezero
\newcommand{\largezero}{\hbox to \wd\sbsmallzero{\hss\raisebox{-\ldiffzero}[0pt][0pt]{\Huge 0}\hss}}

\begin{document}

\title{Generically transitive actions on multiple flag varieties}
\author{Rostislav Devyatov}
\address{Chair of Higher Algebra, Faculty of Mechanics and Mathematics, Lomonosov
Moscow State University, Leninskie Gory 1, GSP-1, Moscow, 119991, Russia}
\email{deviatov@mccme.ru}
\date{}

\begin{abstract}
Let $G$ be a semisimple algebraic group whose decomposition into a product of simple
components does not contain simple groups of type $A$, and $P\subseteq G$ be a parabolic subgroup.
Extending the results of Popov \cp, we enumerate all triples $(G, P, n)$ such that (a) there exists an
open $G$-orbit
on the \textit{multiple flag variety} $G/P\times G/P\times\ldots\times G/P$ ($n$ factors), (b)
the number of $G$-orbits on the multiple flag variety is finite.
\end{abstract}

\maketitle

\section*{Introduction}

Let $G$ be a semisimple connected algebraic group over an algebraically closed field
of characteristic zero, and $P\subseteq G$ be a parabolic subgroup. One easily checks
that the $G$-orbits on $G/P\times G/P$ are in bijection with $P$-orbits on $G/P$. The Bruhat
decomposition of $G$ implies that the number of $P$-orbits on $G/P$ is finite and that these orbits
are enumerated by a subset in the Weyl group $W$ corresponding to $G$.
In particular,
there is an open $G$-orbit on $G/P\times G/P$. So we come to the following questions:
for which $G$, $P$ and $n\ge 3$ is there an open $G$-orbit on the \textit{multiple flag variety}
$(G/P)^n:=G/P\times G/P\times\ldots\times G/P$? For which $G$, $P$ and $n$ is the number of
orbits finite?

Notice that if $G$ is locally isomorphic to $G^{(1)}\times\ldots\times G^{(k)}$, where $G^{(i)}$ are
simple, then there exist parabolic subgroups $P^{(i)}\subseteq G^{(i)}$ such that
$G/P\cong G^{(1)}/P^{(1)}\times\ldots\times G^{(n)}/P^{(n)}$. Hence in the
sequel we may assume that $G$ is simple.
Moreover, let $\pi\colon\widetilde G\to G$ be a simply connected cover. Then $\pi$
induces a bijection between parabolic subgroups $P\subseteq G$ and $\widetilde P\subseteq
\widetilde G$, namely $\widetilde P=\pi^{-1}(P)$, and an isomorphism
$\widetilde G/\widetilde P\to G/P$. Also, $\widetilde G/\widetilde P$ may be considered as
$G$-variety since $\Ker\pi$ acts trivially on it.
In this sense the isomorphism is $G$-equivariant. Therefore we may consider only one
simple group of each type.

The classification of multiple flag varieties with an open $G$-orbit for maximal subgroups
$P$ was given by Popov in \cp. We need some notation to formulate his result.
Fix a maximal torus in $G$ and an associated simple root system $\{\alpha_1, \ldots,
\alpha_l\}$ of the Lie algebra $\mathfrak g=\Lie G$. We enumerate simple roots
as in \cite{Humps}. Let $P_i\subset G$ be the maximal parabolic subgroup corresponding to the
simple root $\alpha_i$.

\begin{theorem}
\cite[Theorem 3]{Popov}\label{popovsresult}
Let $G$ be a simple algebraic group. The diagonal $G$-action on the multiple flag variety
$(G/P_i)^n$ is generically transitive if and only if $n\le 2$ or $(G, n, i)$ is an entry in the following
table:

$$
\begin{array}{c|c}
\mbox{Type of }G & (n, i)\\
\hline\hline
A_l & n < \frac{(l+1)^2}{i(l+1-i)}\\
\hline
B_l, l\ge 3 &
n = 3, i = 1, l\\
\hline
C_l, l\ge 2 &
n = 3, i = 1, l\\
\hline
D_l, l\ge 4 & 
n = 3, i = 1,l-1, l\\
\hline
E_6 & 
n = 3, 4, i = 1, 6\\
\hline
E_7 & 
n = 3, i = 7\\
\end{array}
$$
\end{theorem}

In \cp, the following question was posed: for which non-maximal parabolic
subgroups $P\subset G$ is there an open $G$-orbit in $(G/P)^n$? We solve
this problem for all simple groups except for those of type $A_l$.

Denote the intersection $P_{i_1}\cap\ldots\cap P_{i_s}$ by $P_{i_1, \ldots, i_s}$. It is
easy to see that $P_{i_1, \ldots, i_s}$ is a parabolic subgroup and that every parabolic
subgroup is conjugated to some $P_{i_1, \ldots, i_s}$.

\begin{theorem}
\label{mymainresult}
Let $G$ be a simple algebraic group which is not locally isomorphic to $SL_{l+1}$,
$P\subset G$ be a non-maximal parabolic subgroup and $n\ge 3$.
Then the diagonal $G$-action on the multiple flag variety
$(G/P_i)^n$ is generically transitive if and only if $n=3$ and $(G, P)$ is one of the following pairs:

$$
\begin{array}{c|c}
\mbox{Type of }G & P\\
\hline\hline
D_l, l\ge 5\mbox{ is odd} & P_{1, l-1},P_{1,l}\\
\hline
D_l, l\ge 4\mbox{ is even} & P_{1, l-1}, P_{1, l},P_{l-1, l}\\
\end{array}
$$
\end{theorem}

Now let us consider actions with a finite number of orbits. Recall that a $G$-variety
$X$ is called \textit{spherical} if a Borel subgroup $B\subseteq G$ acts on $X$ with an open orbit.
It is well-known that the number of $B$-orbits on a
spherical variety is finite, see \cite{Bri}, \cite{Vin}.
Equivalently, the number of $G$-orbits on $G/B\times X$
is finite if $X$ is spherical. Therefore, if $P\subseteq G$ is a parabolic subgroup and
$X$ is a spherical $G$-variety, then the number of $G$-orbits on $G/P\times X$ is finite.
The classification of all pairs of parabolic subgroups $(P, Q)$ such that $G/P\times G/Q$ is
spherical is given in \cite{Litt} and \cite{Stem}.
According to this classification, if $(G, P_i)$ is an entry in the table from
Theorem \ref{popovsresult}, then $G/P_i\times G/P_i$ is spherical and hence the number of
$G$-orbits on $G/P_i\times G/P_i\times G/P_i$ is finite. In the last section we
prove that the number of $G$-obits on $(G/P)^n$ is infinite if $n\ge 4$. We also
check directly that if $(G, P)$ is an entry in the table from Theorem
\ref{mymainresult}, then the number of $G$-orbits on $G/P\times G/P\times G/P$ is infinite.
Thus we come to the following result.

\begin{theorem}\label{mysecondaryresult}
Let $G$ be a simple algebraic group, $P\subset G$ be a parabolic subgroup and $n\ge 3$.
The following properties are equivalent.
\begin{enumerate}
\item The number of $G$-orbits on $(G/P)^n$ is finite.
\item $n=3$, $P$ is maximal, and there is an open $G$-orbit on $G/P\times G/P\times G/P$.
\item $n=3$, and $G/P\times G/P$ is spherical.
\end{enumerate}
\end{theorem}

\begin{corollary}
Let $n\ge 3$.
The number of $G$-orbits on $(G/P)^n$ is finite if and only if $n=3$ and $(G, P)$ is one
of the pairs listed in the following table:

$$
\begin{array}{c|c}
\mbox{Type of }G & P\\
\hline\hline
A_l & \mbox{any maximal}\\
\hline
B_l, l\ge 2 & P_1, P_l\\
\hline
C_l, l\ge 3 & P_1, P_l\\
\hline
D_l, l\ge 4 & P_1, P_{l-1}, P_l\\
\hline
E_6 & P_1, P_6\\
\hline
E_7 & P_7\\
\end{array}
$$
\end{corollary}

Let us mention a more general result for classical groups.
Let $Q_{(1)}, \ldots, Q_{(n)}$ be parabolic subgroups in $G$. We call the variety
$G/Q_{(1)}\times\ldots\times G/Q_{(n)}$ a \textit{generalized multiple flag variety}. The classification
of all generalized multiple flag varieties with a finite number of $G$-orbits is
given in \cite{mwz} for $G=SL_{l+1}$ and in \cite{mwzsp} for $G=Sp_{2l}$.

Proofs of Theorems 2 and 3 use methods developed in \cp. The results concerning existence
of an open orbit in a linear representation space in Section \ref{sectopenorbit} may be of independent interest.
In several cases for $G=SO_{2l}$ the existence of an open orbit on a multiple flag variety
is checked directly.

\subsection*{Acknowledgements.} I am grateful to my scientific advisor Ivan V.
Arzhantsev for posing the problem and for paying attention to my work. Also I thank
professor Vladimir L. Popov for useful comments and Dmitri A. Timashev for bringing 
my attention to the articles \cite{Litt} and \cite{Stem}.

\section{Preliminaries}

Let $G$ be a connected simple algebraic group over an algebraically closed field $\mathbb K$
of characteristic zero and $\mathfrak g=\Lie G$. Fix a Borel subgroup $B\subset G$ and a
maximal torus $T\subset B$. These data determine a root system $\Phi$ of $\mathfrak g$,
a positive root subsystem $\Phi^+$ and a system of simple roots $\Delta\subseteq\Phi^+$,
$\Delta=\{\alpha_1, \ldots, \alpha_l\}$. Choose a corresponding Chevalley basis
$\{x_i, y_i, h_i\}$ of $\mathfrak g$. We have $[h, x_i]=\alpha(h)x_i$, $[h, y_i]=-\alpha(h)y_i$
for all $h\in\mathfrak t=\Lie T$ and $h_i=[x_i,y_i]$.

Let $I=\{\alpha_{i_1},\ldots,\alpha_{i_s}\}\subseteq\Delta$ be a subset. The Lie algebra of
the parabolic subgroup $P_I:=P_{i_1,\ldots,i_s}$ is
$$
\mathfrak p=\mathfrak b\oplus\bigoplus_{\alpha\in\Phi_I}\mathfrak g_\alpha,
$$
where $\mathfrak b=\Lie B$ and $\Phi_I\subseteq \Phi^-$ denotes the set of the negative roots
such that their decomposition into the sum of simple roots \textit{does not} contain
the roots $\alpha_i$, $i\in I$. For example, $P_\Delta=B$ and $P_\varnothing=G$.
It is known that \cite[Theorem 30.1]{Hump} if a parabolic group $P$ contains $B$, then $P=P_I$ for some
$I\subseteq\Delta$. Therefore any parabolic subgroup $P\subseteq G$ is conjugate to some
$P_I$. 
If $P=P_I$ for some $I\subseteq\Delta$, we denote by $P^-$ the parabolic subgroup
whose Lie algebra is
$$
\mathfrak p^-=\mathfrak t\oplus
\bigoplus_{\alpha\in-\Phi_I\cup\Phi^-}\mathfrak g_\alpha.
$$

Denote the weight lattice of $T$ by $\mathfrak X(T)$. Let $\mathfrak X^+(T)$ be the subsemigroup
of dominant weights with respect to $B$.
Assume first that $G$ is simply connected. Then $\mathfrak X^+$ is generated by the fundamental weights
$\pi_1,\ldots, \pi_l$. Given a dominant weight $\lambda$, denote the simple
$G$-module with the highest weight $\lambda$ by $V(\lambda)$. If $G$ is not simply connected,
we may consider a simply connected cover $p\colon\widetilde G\to G$, the dominant weight
lattice $\mathfrak X^+(p^{-1}(T))$ and the highest weight $\widetilde G$-module $V(\lambda)$.

Let $G$ be a simple group and $P=P_{i_1,\ldots, i_s}$ be a parabolic subgroup. Notice that
if there is an open $G$-orbit on $(G/P)^n$, then there exists an open $G$-orbit
on $(G/P_i)^n$ for all $i\in\{i_1,\ldots, i_s\}$. Indeed, since $P\subseteq P_i$,
one has the surjective $G$-equivariant map $G/P\to G/P_i$, $gP\mapsto gP_i$. It induces
the surjective $G$-equivariant map $\varphi\colon (G/P)^n\to (G/P_i)^n$, and the image of
an open $G$-orbit on $(G/P)^n$ under $\varphi$ is an open $G$-orbit on $(G/P_i)^n$.
Similarly, if $G$ acts on $(G/P)^n$ with an open orbit and $m<n$, then $G$ acts on $(G/P)^m$
with an open orbit.

Theorem 1 leaves us very few cases of non-maximal parabolic groups to consider.
Namely, if $n>3$ and $G$ is of type $B_l$, $C_l$ or $D_l$, then $G$ never acts on $(G/P)^n$ with
an open orbit. If $n=3$ and $G$ is of type $B_l$ or $C_l$, it suffices to consider
$P=P_{1,l}$, and we show that there is no open orbit in this case. If $n=3$ and $G$ is of type
$D_l$, an open orbit may exist only if $P=P_I$ where $I\subseteq
\{\alpha_1, \alpha_{l-1}, \alpha_l\}$. So there are four cases to consider.
We reduce the case $P_{1,l-1}$ to the case $P_{1,l}$. If $G$ is of type $E_6$,
the only parabolic group we should consider is $P=P_{1,6}$. We show that there is
no open orbit for $n=3$. If $G$ is of type $E_7$, if there existed an open $G$-orbit on
$(G/P)^n$ for $n\ge 3$, then the only maximal parabolic subgroup containing $P$ would be
$P_7$, but in this case $P$ should be maximal itself. If $G$ is of type $E_8$, $F_4$ or
$G_2$, an open orbit exists for no maximal parabolic subgroups for $n\ge 3$, so there are no cases
to consider.

Given a group $G$ acting on an irreducible variety $X$ with an open orbit, according to \cp\
we denote the maximal $n$ such that there is an open $G$-orbit on $X^n$ by
$\gtd (G:X)$. If $G$ acts on $X^n$ with an open orbit, we say that the action
$G:X$ is \textit{generically $n$-transitive}.

We make use of the following fact proved by Popov.
\begin{proposition}
\cite[Corollary 1 (ii) of Proposition 2]{Popov}\label{techprop}
Let $G$ be a simple algebraic group, $P$ be a parabolic subgroup, $P^-$ be an opposite
parabolic subgroup, $L=P\cap P^-$ be the corresponding Levi subgroup
and $\mathfrak u^-$ be the Lie algebra of the unipotent radical of $P^-$.
If $P$ is conjugate to $P^-$, then $\gtd(G:G/P)=2+\gtd(L:\mathfrak u^-)$.
\end{proposition}

We suppose that the group $SO_l$ acts in the $l$-dimensional space and preserves the
bilinear form whose matrix with respect to a standard basis is
$$
Q=\left(
\begin{array}{cccccc}
&&&&&1\\
&
\largezero
&&&\udots&\\
&&&1&&\\
&&1&&&\\
&\udots&&&\largezero&\\
1&&&&&
\end{array}
\right).
$$

We denote the $l$-dimensional projective space by $\mathbf P^l$ and
the Grassmannian of $k$-dimensional subspaces in $\mathbb K^l$ by $\mathrm{Gr}(k,l)$.

\section{Existence of an open orbit}\label{sectopenorbit}

\subsection{Groups of type $B_l$}
By Theorem \ref{popovsresult}, it is sufficient to consider the case $P=P_{1,l}$.
The Dynkin diagram $B_l$ has no automorphisms,
hence $P$ is conjugate to $P^-$. So we may apply Proposition \ref{techprop}, and it
suffices to check that $\gtd(L:\mathfrak u^-)=0$, i. e. $L$ acts on $\mathfrak u^-$ with
no open orbit.


Let $G=SO_{2l+1}$. Then $L=\mathbb K^*\times GL_{l-1}$ and the $L$-module $\mathfrak u^-$ can
be decomposed into the direct sum $V_1\oplus V_2\oplus V_3\oplus V_4\oplus V_5$. Here
$V_1$ is 
a $GL_{l-1}$-module $(\mathbb K^{l-1})^*$ dual to the tautological one
and its $\mathbb K^*$-weight is 1,
$V_2$ is a trivial one-dimensional $GL_{l-1}$-module of weight 1,
$V_3$ is a tautological $GL_{l-1}$-module $\mathbb K^{l-1}$ of weight 0,
$V_4$ is a $GL_{l-1}$-module $\mathbb K^{l-1}$ of weight 1,
$V_5$ is 
a $GL_{l-1}$-module $\Lambda^2\mathbb K^{l-1}$ 
of weight 0.
According to this decomposition, we denote components of a vector $u\in \mathfrak u^-$
by $u_1, u_2, u_3, u_4, u_5$.

Notice that there exists a $GL_{l-1}$-invariant pairing between $V_1$ and $V_3$.
Its $\mathbb K^*$-weight is 1. Also there exists a $GL_{l-1}$-invariant pairing
between $V_1$ and $V_4$, whose $\mathbb K^*$-weight is 2.
Therefore the rational function
$$
\frac{(u_1, u_3)^2}{(u_1, u_4)}
$$
is a non-constant invariant for $L:\mathfrak u^-$, and the action of $G$ on $G/P$ is not
generically 3-transitive.

\subsection{Groups of type $C_l$}
This case is completely similar to the previous one, and again the only thing we should do is to
prove that there is no open $L$-orbit on $\mathfrak u^-$, where $L=P\cap P^-$ is a Levi subgroup
of $P=P_{1,l}$ and $\mathfrak u^-$ is the Lie algebra of the unipotent radical of $P^-$.


Let $G=Sp_{2l}$. Then $L=\mathbb K^*\times GL_{l-1}$ and the $L$-module $\mathfrak u^-$ can
be written as
$V_1\oplus V_2\oplus V_3\oplus V_4$. Here
$V_1$ is 
a $GL_{l-1}$-module $(\mathbb K^{l-1})^*$ 
and its $\mathbb K^*$-weight is 1,
$V_2$ is a $GL_{l-1}$-module $\mathbb K^{l-1}$ of weight 1,
$V_3$ is 
a $GL_{l-1}$-module $S^2\mathbb K^{l-1}$
of weight 0,
$V_4$ is a trivial $GL_{l-1}$-module of weight 2.
According to this decomposition, we denote components of a vector $u\in \mathfrak u^-$
by $u_1, u_2, u_3, u_4$.

We see that there exists a $GL_{l-1}$-invariant pairing between $V_1$ and $V_2$ 
with $\mathbb K^*$-weight is 2. Therefore we have the rational invariant
$$
\frac{(u_1, u_2)}{u_4}
$$
for $L:\mathfrak u^-$, and the action of $G$ on $G/P$ is not generically 3-transitive.

\subsection{Groups of type $D_l$}
This time we should consider the following four cases of parabolic subgroups: $P=P_{1,l-1},
P_{1,l}, P_{l-1,l}, P_{1,l-1,l}$. 
One easily checks that $P$ and $P^-$ are conjugate except for the cases $P=P_{1,l}$, $l$ odd,
and $P=P_{1,l-1}$, $l$ odd.

Let $G=SO_{2l}$. There exists a diagram automorphism of $G$ that interchanges $\alpha_{l-1}$
and $\alpha_l$. It preserves the maximal torus and the Borel subgroup and interchanges $P_{1,l-1}$
and $P_{1,l}$. Therefore, the actions $G:G/P_{1,l-1}$ and $G:G/P_{1,l}$ are either
generically 3-transitive or not generically 3-transitive simultaneously.

\subsubsection{$P=P_{l-1,l}$} In this case, $P$ and $P^-$ are conjugate, and we have to
find $\gtd(L:\mathfrak u^-)$.


The Levi subgroup $L$ is isomorphic to $\mathbb K^*\times GL_{l-1}$ and the $L$-module
$\mathfrak u^-$ is isomorphic to $V_1\oplus V_2\oplus V_3$, where
$V_1$ is 
a $GL_{l-1}$-module $\Lambda^2\mathbb K^{l-1}$
and its $\mathbb K^*$-weight is 0,
$V_2$ is a $GL_{l-1}$-module $\mathbb K^{l-1}$ of weight 1,
$V_3$ is a $GL_{l-1}$-module $\mathbb K^{l-1}$ of weight $-1$.
We denote components of a vector $u\in \mathfrak u^-$
by $u_1, u_2, u_3$.

\underline{Let $l$ be odd.} Then a generic element $u_1\in V_1$ gives rise to a
non-degenerate skew-symmetric bilinear form on the $GL_{l-1}$-module $(\mathbb K^{l-1})^*$.
Furthermore, one can consider the corresponding skew-symmetric
form on the tautological $GL_{l-1}$-module. This form is obtained by matrix inversion
and we denote it by $u_1^{-1}$. The following function is a rational $L$-invariant:
$$
u_1^{-1}(u_2, u_3).
$$
Thus the action of $G$ on $G/P$ is not generically 3-transitive.

\underline{Let $l$ be even.} We prove that there is an open $L$-orbit
on $\mathfrak u^-$.

Consider the $GL_{l-1}$-module $V'=V_1\oplus V_2$, where $V_2$ is a 
$GL_{l-1}$-module $\mathbb K^{l-1}$ and $V_1=\Lambda^2\mathbb K^{l-1}$. 

Since $l-1$ is odd, the rank of a generic element $w\in V_1$ is $l-2$. Denote the set
of all $w\in V_1$ such that $\rk w=l-2$ by $Z$. Any element
$w\in Z$ gives rise to a (degenerate) skew-symmetric form on $V_2^*$, and
$\dim\Ker w=1$. 
Consider the subspace $(\Ker w)^\bot\subset V_2$ where all
the functions from the kernel vanish. Denote $V_2\setminus (\Ker w)^\bot$ by $X_w$.
Clearly, $W_1=\cup_{w\in Z}(w\times X_w)$ is an open $GL_{l-1}$-invariant subset
of $V'$.


Let us prove that $GL_{l-1}$ acts transitively on $W_1$.
First, given an element $u=(u_1,u_2)\in W_1$, one can apply an element of $GL_{l-1}$ such that the matrix of
the bilinear form $u_1$ in the corresponding basis is
$$
R=
\left(
\begin{array}{cccccc}
0&&&&&\\
&0&1&&\largezero&\\
&-1&0&&&\\
&&&\ddots&&\\
&\largezero&&&0&1\\
&&&&-1&0
\end{array}
\right).
$$

The first coordinate of $u_2$ in the new basis is non-zero since $u_2\in X_{u_1}$. Denote
the $i$-th coordinate of $u_2$ by $(u_2)_i$. The following
element of $GL_{l-1}$ preserves the bilinear form with matrix $R$:
$$
\left(
\begin{array}{cccc}
1&0&\ldots&0\\
-(u_2)_3/(u_2)_1\\
\vdots&&I_{l-2}\\
-(u_2)_{l-1}/(u_2)_1
\end{array}
\right).
$$
When we apply it to $u_2$, all its coordinates will be zero except for the first one.

So any element of $W_1$ can be transformed by $GL_{l-1}$-action to 
an element of the form $u_1=R$, $u_2=((u_2)_1, 0,
\ldots, 0)^T$, where $(u_2)_1\ne 0$. Clearly, all these elements belong to
the same $GL_{l-1}$-orbit. Call such an element of $V'$ \textit{canonical}, 
i. e. call an element $(u_1, u_2)\in V'$ \textit{canonical} if $u_1=R$, $u_2=((u_2)_1, 0,
\ldots, 0)^T$, where $(u_2)_1\ne 0$.
The stabilizer of $(u_1,\langle u_2\rangle)$ consists of 
direct sums of a non-zero $1\times 1$ matrix and
a symplectic  $(l-2)\times (l-2)$ matrix. Such an element fixes $u_2$ as well if and only
if the first $1\times 1$ matrix is 1.

Now we are ready to consider the $L$-action on $\mathfrak u^-$. Maintain the above notation.
Since $V_2$ and $V_3$ are isomorphic as $GL_{l-1}$-modules, 
for each $w\in Z\subset V_1$ we can similarly consider the open subset 
$V_3\setminus (\Ker w)^\bot$. Denote it by $Y_w$. Define
the subsets $W_2=\cup_{w\in Z}(w\times X_w\times Y_w)$ and
$W=\{u\in W_2:u_2\mbox{ is not a multiple of }u_3\}$. Let us prove that $L$ acts on $W$
transitively.

We may suppose that $u_1$ and $u_2$ are canonical in the sense stated above.
Applying a diagonal matrix from $(GL_{l-1})_{u_1,\angs{u_2}}$, we may assume that
$(u_2)_1(u_3)_1=1$ since $V_2$ and $V_3$ are both tautological $GL_{l-1}$-modules. Since
$u_3$ is not a multiple of $u_2$, the vector $v=((u_3)_2, (u_3)_3, \ldots, (u_3)_{l-1})$ is
not zero. Since $Sp_{l-2}$ acts transitively on $\mathbb K^{l-2}\setminus 0$, there exists
an element $g\in (GL_{l-1})_{u_1, u_2}$, $g=g_1\oplus g_2$, $g_1=1$, $g_2\in Sp_{l-2}$
such that $g_2v=((u_3)_1, 0, \ldots, 0)^T$. In other words, we may suppose that
$u_1$ and $u_2$ are canonical, the only non-zero coordinates of $u_3$ are the first one
and the second one, they are equal, and $(u_2)_1(u_3)_1=1$.

Now recall that $L=GL_{l-1}\times \mathbb K^*$, the $\mathbb K^*$-weights of $V_1$, $V_2$ and $V_3$
are 0, 1 and $-1$, respectively. Therefore, after applying a suitable element of $\mathbb K^*$,
we have $u_1=S$, $u_2=(1, 0, \ldots, 0)^T$ and $u_3=(1, 1, 0, \ldots, 0)^T$, and the $G$-action
on $G/P$ is generically 3-transitive.

\subsubsection{$P=P_{1, l}$.} In this case, Proposition \ref{techprop} applies if and
only if $l$ is even.

\underline{Let $l$ be even.} It is sufficient to prove that there is an open $L$-orbit
on $\mathfrak u^-$.


Again $L=\mathbb K^*\times GL_{l-1}$, and the $L$-module $V$ can be decomposed into three
summands, $V=V_1\oplus V_2\oplus V_3$, but this time
$V_1$ is 
a $GL_{l-1}$-module $\Lambda^2\mathbb K^{l-1}$
and its $\mathbb K^*$-weight is 0,
$V_2$ is a $GL_{l-1}$-module $\mathbb K^{l-1}$ of weight 1,
$V_3$ is 
a $GL_{l-1}$-module $(\mathbb K^{l-1})^*$ 
of weight 1.
We denote components of an element $u\in \mathfrak u^-$ by $u_1, u_2, u_3$.

Recall the notation we have introduced for the $GL_{l-1}$-module $V'$. Also this time
denote $Y_w=V_3\setminus\Ker w$. Define $W_2=\cup_{w\in Z}(w\times X_w\times Y_w)$ and
$W=\{u\in W_2: \langle u_2, u_3\rangle\ne 0\}$. Here $\langle\cdot,\cdot\rangle$ denotes
the $GL_{l-1}$-invariant pairing between $V_2$ and $V_3$. Its $\mathbb K^*$-weight is 2, but
the condition $\langle u_2, u_3\rangle\ne 0$ is not affected by $\mathbb K^*$-action, so $W$
is $L$-invariant. We are going to prove that $L$ acts transitively on $W$.

Again we may suppose that $u_1=R$ and the only non-zero coordinate of $u_2$ is the
first one. Notice that $(u_3)_1\ne 0$ since $\langle u_2, u_3\rangle\ne 0$.
This time $V_2$ and $V_3$ are dual $GL_{l-1}$-modules, so by applying a suitable element of
$(GL_{l-1})_{u_1, \angs{u_2}}$ the coordinates $(u_2)_1$ and $(u_3)_1$ can be made equal.

Consider the vector $v=((u_3)_2, (u_3)_3, \ldots, (u_3)_{l-1})$. It cannot be zero since
$u_3\notin\Ker u_1$.
Since $Sp_{l-2}$ acts transitively on $(\mathbb K^{l-2})^*\setminus 0$, there exists
an element $g\in (GL_{l-1})_{u_1, u_2}$, $g=g_1\oplus g_2$, $g_1=1$, $g_2\in Sp_{l-2}$
such that $g_2v=((u_3)_1, 0, \ldots, 0)^T$. In other words, we may suppose that
$u_1$ and $u_2$ are canonical, the only non-zero coordinates of $u_3$ are the first one
and the second one, and $(u_2)_1=(u_3)_1=(u_3)_2$.

This time the $\mathbb K^*$-weights of $V_2$ and $V_3$ are both 1, so with the help of
the $\mathbb K^*$-action we can satisfy the equality $(u_2)_1=(u_3)_1=(u_3)_2=1$. Thus,
$L$ acts transitively on $W$, and the $G$-action on $G/P$ is generically 3-transitive.

\underline{Let $l$ be odd.} Proposition \ref{techprop} does not apply, and we have to
find $\gtd(G:G/P)$ directly.

Consider the tautological $SO_{2l}$-module $\mathbb K^{2l}$, and let $e_1, \ldots, e_{2l}$
be its standard basis.
Let $X'\subset \mathrm{Gr}(l, 2l)$ be the set of all isotropic subspaces of dimension $l$ in
$\mathbb K^{2l}$. One easily checks that $X'$ is a disjoint
union of two $SO_{2l}$-orbits, and the group $O_{2l}$ interchanges them. If two subspaces
belong to the same $SO_{2l}$-orbit, then their intersection is non-zero.

Denote the orbit $SO_{2l}\angs{e_1, \ldots, e_l}\subset X'$ by $X$. Then $X$ is an
irreducible subvariety in $\mathrm{Gr}(l, 2l)$.

For each $s\in X$ let $Y_s\subset\mathbf P^{2l-1}$ be the set of all lines contained in $s$.
Clearly, $W=\cup_{s\in X}(s\times Y_s)$ is a closed $G$-invariant subset in
$\mathrm{Gr} (l, 2l)\times\mathbf P^{2l-1}$. One easily checks that $G/P=W$.

Let us prove that there exists an open $G$-orbit on $W\times W\times W$. We impose
some conditions on the point $(s_1, a_1, s_2, a_2, s_3, a_3)\in W\times W\times W$ and so define
an open subset $Y\subseteq W\times W\times W$. Then we define a point $p\in \mathrm{Gr} (l, 2l)\times\mathbf P^{2l-1}
\times \mathrm{Gr} (l, 2l)\times\mathbf P^{2l-1}\times \mathrm{Gr} (l, 2l)\times\mathbf P^{2l-1}$ and
prove that (a) each point $y\in Y$
belongs to the same $G$-orbit that $p$ does, and (b) $p$ belongs to $Y$. 
Condition (b) guarantees that $Y$ is not empty.

Let $Y\subseteq W\times W\times W$ be the set of all tuples $(s_1, a_1, s_2, a_2, s_3, a_3)$
such that:
\begin{enumerate}
\item $s_1\cap s_2\cap s_3=0$.
\item $s_1+s_2+s_3=\mathbb K^{2l}$.
\item $\dim s_1\cap s_2=\dim s_2\cap s_3=\dim s_1\cap s_3=1$.
\item $\dim (a_1+a_2+a_3)=3$.
\item The intersection of the subspaces $s=(s_1\cap s_2)+(s_2\cap s_3)+(s_1\cap s_3)$ and
$a=a_1+a_2+a_3$ is zero.
\item $a_i+s_j+s_k=\mathbb K^{2l}$, where $i=1, 2, 3, j\ne i, k\ne i, j<k$.
\item The lines $a_i$ and $a_j$ are not orthogonal for all $i\ne j$.
\end{enumerate}
Notice that if conditions (1)--(3) hold, the sum of subspaces $s_1\cap s_2$, $s_2\cap s_3$
and $s_1\cap s_3$ is direct.

Let us prove that $G$ acts transitively on $Y$. Choose vectors $f_1, f_2, f_3$ such that
$\angs{f_i}=a_i$, and vectors $f_4, f_5, f_6$ such that $\angs{f_4}=s_2\cap s_3$,
$\angs{f_5}=s_1\cap s_3$, $\angs{f_6}=s_1\cap s_2$. The restriction of the bilinear
form to the subspace $S=\angs{f_1, \ldots, f_6}$ is defined by the following matrix:
$$
\left(
\begin{array}{cccccc}
0&b_1&b_2&b_4&&\\
b_1&0&b_3&&b_5&\\
b_2&b_3&0&&&b_6\\
b_4&&&&&\\
&b_5&&&&\\
&&b_6&&&
\end{array}
\right).
$$
Conditions (6) and (7) imply that $b_i\ne 0$  for all $i$. Clearly, this matrix is
non-degenerate.

The above choice of the vectors $f_i$ allows to multiply them by scalars. 
Up to scalar multiplication we may assume that all $b_i=1$.

Notice that a cyclic permutation of $f_1, f_2, f_3$ and the same permutation
of $f_4, f_5, f_6$ performed simultaneously define a linear operator on $S$ that
preserves the restriction of the bilinear form and whose determinant is 1.

Consider the following basis of $S$:
$g_1=f_1$,
$g_2=f_5$,
$g_3=f_6$,
$g_4=f_3-f_4-f_5$,
$g_5=f_2-f_4$,
$g_6=f_4$.
One checks directly that the matrix of the bilinear form with respect to this basis is $Q$.
Obviously, there exists a matrix $M$ such that $(f_1,\ldots, f_6)=(g_1,\ldots, g_6)M$
and whose elements do not depend on $a_i$ and $s_i$.

The restriction of the bilinear form to $S$ is non-degenerate, hence its restriction to
$S^\bot$ is also non-degenerate. Since $s_i=s_i^\bot$, $\dim (s_i\cap S^\bot)=l-3$ for all
$i$.

Thus, $S^\bot$ is a subspace of even dimension equipped with a non-degenerate symmetric bilinear
form. We have three isotropic subspaces of maximal dimension in $S^\bot$, and the intersection of
any two of them is zero. Let us prove the following lemma.

\begin{lemma}
Let $(\cdot,\cdot)$ be a non-degenerate symmetric bilinear form in $\mathbb K^{2k}$,
and $U_1, U_2, U_3$ be isotropic subspaces of dimension $k$ with $U_i\cap U_j=0$ for $i\ne j$.
Then there exist matrices $M_1, M_2, M_3\in Mat_{2k\times k}$ that do
not depend on $U_i$ and a basis $e_1, \ldots, e_{2k}$ of $\mathbb K^{2k}$ such that:
(a) the matrix of the bilinear form is $Q$ and (b) $(e_1, \ldots, e_{2k})M_i$ is a
basis of $U_i$.
\end{lemma}

\begin{proof}
Consider the non-degenerate linear map $A\colon U_1\to U_2$ whose graph is the
subspace $U_3$. This is possible since $U_i\cap U_j\ne 0$ for $i\ne j$. In terms
of the map, $U_3=\{v+Av\mid v\in U_1\}$.

Consider the bilinear form $(v_1, v_2)_A=
(v_1, Av_2)$ on $U_1$. 
Since $U_3$ is isotropic, we have
$0=(v_1+Av_1, v_2+Av_2)=(v_1,v_2)+(Av_1+Av_2)+(v_1, Av_2)+(Av_1, v_2)=
(v_1,Av_2)+(v_2,Av_1)=(v_1,v_2)_A+(v_2,v_1)_A$ for all $v_1, v_2\in U_1$, 
hence the form $(\cdot, \cdot)_A$
is skew-symmetric. Assume that it is degenerate and $v\in U_1$ belongs to its kernel.
Then $(v_1,v)=(v_1,Av)=0$ for all $v_1\in U_1$. Since the pairing between
trivially intersecting isotropic subspaces $U_1$ and $U_2$ of maximal dimension
is non-degenerate, $Av=0$. Since $\Ker A=0$, $v=0$ and the form $(\cdot,\cdot)_A$
is non-degenerate.

Thus, we have a symplectic space $U_1$ with the skew-symmetric form $(\cdot,\cdot)_A$.
Hence $k$ is even. Choose a basis $\angs{q_1, \ldots, q_k}$ of $U_1$ such that the
matrix of the skew-symmetric form is
$$
\left(
\begin{array}{cccccc}
&&&&&1\\
&\largezero&&&\udots&\\
&&&1&&\\
&&-1&&&\\
&\udots&&&\largezero&\\
-1&&&&&
\end{array}
\right).
$$

The vectors $q_1, \ldots, q_k$ are linearly independent, so let them be the
first $k$ elements of a basis of $\mathbb K^{2l}$. Define the rest of the basis as follows:
$q_{k+j}=-Aq_j$ if $j=1, \ldots, k/2$ and $q_{k+j}=Aq_j$ if $j=k/2+1, \ldots, k$.
The matrix of the bilinear form $(\cdot, \cdot)$ is $Q$. The subspaces
$U_i$ have the following bases:
$$
\begin{aligned}
U_1&=\angs{q_1, \ldots, q_k}\\
U_2&=\angs{q_{k+1}, \ldots, q_{2k}}\\
U_3&=\langle q_1-q_{k+1}, \ldots, q_{k/2}-q_{k+k/2}, q_{k/2+1}+q_{k+k/2+1},\ldots, q_k+q_{2k}\rangle.\\
\end{aligned}
$$
This completes the proof of the lemma.
\end{proof}

Consider the following basis of $\mathbb K^{2l}$: $g_1, g_2, g_3, q_1, \ldots,
q_{2l-6}, g_4, g_5, g_6$, where $q_i$ are defined above in the proof of the lemma.
Notice that the matrix of the bilinear form in this basis is $Q$.
Define the operator $B\colon\mathbb K^{2l}\to\mathbb K^{2l}$ that maps this basis
to the standard one. We know that the matrix of the bilinear form is $Q$ in both bases,
so $B\in O_{2l}$. 

Let us check that $\det B=1$. Assume that $\det B=-1$.
Since $s_1=\angs{g_1, g_2, g_3, q_1, \ldots, q_{l-3}}$,
$Bs_1=\angs{e_1,\ldots, e_l}\in X$. Since $s_1\in X$, there exists an operator
$C\in SO_{2l}$ such that $CBs_1=s_1$. Thus, $CB\in (O_{2l})_{s_1}$, $\det CB=-1$,
and the $O_{2l}$-orbit $X'$ cannot be a union of two distinct $SO_{2l}$-orbits,
a contradiction.

Bases of the subspaces $Ba_i$ and $Bs_i$ can be written in terms of $e_i$ using
matrices that do not depend on $a_i$ and $s_i$. Namely, they are the same matrices
that we need to write bases of $a_i$ and $s_i$ using $g_i$ and $q_i$, and the latter
do not depend on $a_i$ and $s_i$.
Denote the 6-tuple $(Bs_1,Ba_1,Bs_2,Ba_2,Bs_3,Ba_3)$ by $p$.
It suffices to prove that 
$p\in Y$. Conditions (1)--(7) hold by the construction of
$g_i$ and $q_i$, but we should check that $(Bs_1, Bs_2, Bs_3)\in X\times X\times X$.
It is sufficient to find elements of $SO_{2l}$ that map $s_1$ to $s_2$ and $s_1$ to
$s_3$. Since $s_i=(s_i\cap S)\oplus (s_i\cap S^\bot)$, we find them as direct sums
of elements of $SO(S)$ and $SO(S^\bot)$. The elements of $SO(S)$ are already found,
they are cyclic permutations of $f_1, f_2, f_3$ and $f_4, f_5, f_6$. To interchange
$s_1\cap S^\bot$ and $s_2\cap S^\bot$, consider the map that permutes all the pairs of
vectors $g_i\leftrightarrow g_{2l+1-i}$, $i=1,\ldots l-3$. It is orthogonal
and its determinant is 1 since $l-3$ is even. Finally, the
operator with the following matrix in the basis $g_i$ maps $s_1\cap S^\bot$ to $s_3\cap S^\bot$.
$$
\left(
\begin{array}{cc}
I_{l-3}&0\\
D&I_{l-3},
\end{array}
\right)
$$
where
$$
D=\left(
\begin{array}{cccccc}
-1&&&&&\\
&\ddots&&&\largezero&\\
&&-1&&&\\
&&&1&&\\
&\largezero&&&\ddots&\\
&&&&&1
\end{array}
\right).
$$

Therefore, $SO_{2l}$ acts transitively on $Y$, and $\gtd(G:G/P)=3$.


\subsubsection{$P=P_{1,l-1,l}$} The subgroups $P$ and $P^-$ are conjugate for all $l$.
It is sufficient to find $\gtd(L:\mathfrak u^-)$, where $L=(\mathbb K^*)^2\times GL_{l-2}$
and the $L$-module $\mathfrak u^-$ is isomorphic to the direct sum of 7 simple modules
that we denote by $V_1, \ldots, V_7$. Namely,
$V_1$ is 
a $GL_{l-1}$-module $(\mathbb K^{l-2})^*$ 
and its $(\mathbb K^*)^2$-weight is $(1,0)$,
$V_2$ is a trivial $GL_{l-2}$-module of weight $(1,1)$,
$V_3$ is a $GL_{l-2}$-module $\mathbb K^{l-2}$ of weight $(0,1)$,
$V_4$ is a trivial $GL_{l-2}$-module of weight $(1,-1)$,
$V_5$ is a $GL_{l-2}$-module $\mathbb K^{l-2}$ of weight $(0,-1)$,
$V_6$ is a $GL_{l-2}$-module $\mathbb K^{l-2}$ of weight $(1,0)$,
$V_7$ is 
a $GL_{l-2}$-module $\Lambda^2\mathbb K^{l-2}$
of weight $(0,0)$.
Denote the components of $u\in\mathfrak u^-$ by $u_1, \ldots, u_7$.

There exists a $GL_{l-2}$-invariant pairing between $V_1$ and $V_3$ whose $(\mathbb K^*)^2$-weight is
$(1,1)$. The following function is a rational $L$-invariant:
$$
\frac{(u_1, u_3)}{u_4}.
$$

Thus, the $G$-action on $G/P$ is not generically 3-transitive.

\subsection{Groups of type $E_6$} The only parabolic subgroup to consider is $P=P_{1,6}$. The
set $\{1,6\}$ of Dynkin diagram vertices is invariant under all automorphisms of the
Dynkin diagram. Hence the Weyl group element of the maximal length interchanges $P$
and $P^-$. We have to find $\gtd (L:\mathfrak u^-)$.


The Levi subgroup $L$ is locally isomorphic to $(\mathbb K^*)^2\times SO_8$, and the $L$-module
$\mathfrak u^-$ is isomorphic to $V_1\oplus V_2\oplus V_3$. Here $V_1$ is an $SO_8$-module
with the lowest weight $-\pi_1$, i. e. a tautological $SO_8$-module, $V_2$ is an $SO_8$-module
with the lowest weight $-\pi_3$, $V_3$ is an $SO_8$-module with the lowest weight $-\pi_4$.
Denote the components of $u\in\mathfrak u^-$ by $u_1, u_2, u_3$.

Since $V_1$ is a tautological $SO_8$-module, there exists an $SO_8$-invariant symmetric
bilinear form on it that we denote by $(u_1, u_1)$. There exist diagram automorphisms
of $SO_8$ that transform the tautological $SO_8$-module to $SO_8$-modules isomorphic to
$V_2$ and $V_3$. So there exist an $SO_8$-invariant on $V_2$ that we denote by $(u_2, u_2)$
and an $SO_8$-invariant on $V_3$ that we denote by $(u_3, u_3)$. These bilinear forms
are not necessarily $(\mathbb K^*)^2$-invariant, in general their $(\mathbb K^*)^2$-weights
are three pairs of integers. There is a linear combination of these pairs that is equal
to zero. Hence, there exists a non-trivial rational $L$-invariant of the form
$$
(u_1, u_1)^a(u_2, u_2)^b(u_3, u_3)^c,
$$
where $a, b, c\in\mathbb Z$, and the $G$-action on $G/P\times G/P\times G/P$ is not
generically transitive.

\section{Finite number of orbits}

\begin{proposition}
Let $G$ be a simple algebraic group and $P$ be a proper parabolic subgroup. If
$n\ge 4$, the number of $G$-orbits on $(G/P)^n$ is infinite.
\end{proposition}

\begin{proof}
Let $P=P_{i_1,\ldots,i_s}$. Consider the dominant weight $\lambda=\pi_{i_1}+\ldots+\pi_{i_s}$.
Then $G/P$ is isomorphic to the projectivization of the orbit of the highest weight
vector $v_\lambda\in V(\lambda)$. In the sequel we shortly write $i=i_1$. It is easy to
check that $y_i^2v_\lambda=0$. Denote the unipotent subgroup $\exp(ty_i)$ by $U_i$.
We see that $U_iv_\lambda$ is an affine line not containing zero. The closure of
its image in the projectivization $\mathbf P(V(\lambda))$ is a projective line
$\mathbf P^1\subseteq G/P\subseteq\mathbf P(V(\lambda))$. Choose $n\ge 4$ points
$(x_1,\ldots, x_n)\in\mathbf P^1\times\ldots\times\mathbf P^1\subseteq G/P\times\ldots\times
G/P$. The double ratio of the first four of these points does not change under
$G$-action. Hence, two $n$-tuples with different double ratios cannot belong to the
same orbit, and the number of orbits is infinite.
\end{proof}

Now we prove that in the cases $P=P_{1,l}$ and $P=P_{l-1, l}$ the number of orbits on 
$G/P\times G/P\times G/P$ is infinite.

We suppose that $G=SO_{2l}$. Let $\mathbb K^{2l}$ be the tautological $SO_{2l}$-module and let
$e_1, \ldots, e_{2l}$ be the standard basis.
Let $X'\subset \mathrm{Gr}(l, 2l)$ be the set of all isotropic
subspaces of dimension $l$ in $\mathbb K^{2l}$. It is known that $G/P_l$ is isomorphic to a
connected component of $X'$. In the sequel we suppose that $G/P_l=X\subseteq X'$.
For each $s\in X$ let $Y_s\subset\mathbf P^{2l-1}$ be the set of all lines contained in $s$.
One easily checks that the closed subset $Y=\cup_{s\in X}(s\times Y_s)\subset
\mathrm{Gr} (l, 2l)\times\mathbf P^{2l-1}$ is isomorphic to $SO_{2l}/P_{1,l}$.

Similarly, if $s\in X$, denote by $Z_s\subset \mathrm{Gr}(l-1,2l)$ the set of all 
subspaces of dimension $l-1$ in $s$. Let $Z$ be the closed subset
$\cup_{s\in X}(s\times Z_s)\subset \mathrm{Gr} (l, 2l)\times \mathrm{Gr} (l-1, 2l)$.
One easily checks that it is isomorphic to $SO_{2l}/P_{l-1,l}$.

First, let $l=3$. Consider the following isotropic
subspaces: $S_1=\angs{e_1, e_2, e_4}$, $S_2=\angs{e_2, e_3, e_6}$, $S_3=\angs{e_1, e_3, e_5}$.
They belong to the same $SO_6$-orbit, so we may suppose that $S_1, S_2, S_3\in X$.
Choose a line $T_1\subset S_1$ such that $T_1\subset\angs{e_1,e_2}$. Also
choose lines $T_2\subset S_2$ and $T_3\subset S_3$ such that
$T_2\subset\angs{e_2,e_3}$ and $T_3\subset\angs{e_1, e_3}$. Impose one more restriction,
namely, the sum $T_2+T_3$ should be direct and should not be equal to $\angs{e_1, e_2}$. Consider the
point $((S_1, T_1),(S_2, T_2),(S_3, T_3))\in G/P_{1,l}\times G/P_{1,l}\times G/P_{1,l}$.
There are four subspaces of $\angs{e_1,e_2}$: $\angs{e_1}=S_1\cap S_3$,
$\angs{e_2}=S_2\cap S_3$, $T_1$ and $T_4=(T_2\oplus T_3)\cap\angs{e_1, e_2}$.
Thus, we have defined four lines in $\mathbb K^6$ in terms of intersections 
and sums of $S_i$ and $T_i$. If we apply an element $g\in G$ to these four lines, 
we will obtain four lines defined in the same way using $gS_i$ and $gT_i$ 
instead of $S_i$ and $T_i$. 
The double ratio of these four lines in their sum of
dimension two is not changed under $G$-action. Since $T_1$ is chosen arbitrarily,
this double ratio can be any number and the number of orbits is infinite.

Consider the same subspaces $S_i$ and $T_i$ and set $U_1=T_1\oplus\angs{e_4}$,
$U_2=T_2\oplus\angs{e_6}$ and $U_3=T_3\oplus\angs{e_5}$. 
The point
$((S_1, U_1),(S_2, U_2),(S_3, U_3))$ belongs to $Z\times Z\times Z$.
Note that $\angs{e_1, e_2, e_3}=(S_1\cap S_2)\oplus (S_2\cap S_3)\oplus (S_1\cap S_3)$
and $T_i=U_i\cap\angs{e_1, e_2, e_3}$. 
Again we have 
a subspace of dimension two and four lines
in it defined in terms of intersections and sums of $S_i$ and $U_i$. 
The existence of $SO_6$-invariant double ratio in this case
yields that the number of orbits is infinite.

Let $l>3$. Construct the subspaces $S_i$, $T_i$ and $U_i$ as above, using the last three
basis vectors instead of $e_4, e_5, e_6$. Let $S_i'=S_i\oplus\angs{e_4,\ldots,e_l}$ and
$U_i'=U_i\oplus\angs{e_4,\ldots,e_l}$. 
The points
$((S_1',T_1),(S_2',T_2),(S_3',T_3))$
and
$((S_1',U_1'),(S_2',U_2'),(S_3',U_3'))$
belong to $Y\times Y\times Y$ and $Z\times Z\times Z$, respectively.
Consider also the subspace $V=(S_1'\cap S_2'\cap S_3')^\bot$. The restriction of the
bilinear form to this subspace is degenerate, its kernel is $S_1'\cap S_2'\cap S_3'=
\angs{e_4, \ldots, e_l}$. The quotient is a space of dimension 6 with a bilinear form.
The quotient morphism restricted to $\angs{e_1,e_2,e_3,e_{2l-2},e_{2l-1},e_{2l}}$ is
an isomorphism, so we have subspaces $S_i$, $T_i$, $U_i$ in the 6-dimensional space. This
is exactly the same situation as we had above for the group $SO_6$, and it
enables us to define double ratios for the points of $G/P_{1,l}\times G/P_{1,l}\times G/P_{1,l}$
and $G/P_{l-1,l}\times G/P_{l-1,l}\times G/P_{l-1,l}$ under consideration.
Therefore the number of $SO_{2l}$-orbits on these multiple flag varieties is infinite.
This finishes the proof of Theorem \ref{mysecondaryresult}.


\begin{thebibliography}{11}
\bibitem{Bri} M. Brion, \textit{Quelques propri\'et\'es des espaces homog\`enes sph\'eriques},
Manuscripta Math. \textbf{55} (1986), no.~2, 191--198.
\bibitem{Hump} J. Humphreys, \textit{Linear algebraic groups}. GTM \textbf{21} Springer-Verlag,
New York-Heidelberg, 1975.
\bibitem{Humps} J. Humphreys, \textit{Introduction to Lie Algebras and Representation Theory}.
GTM \textbf{9}, Springer-Verlag, New York-Berlin, 1978.
\bibitem{Litt} P. Littleman, \textit{On spherical double cones}, J. Algebra \textbf{166} (1994),
no.~1, 142--157.
\bibitem{mwz} P. Magyar, J. Weynman, A. Zelevinsky, \textit{Multiple flag varieties of finite type},
Adv. Math. \textbf{141} (1999), no. 1, 97--118.
\bibitem{mwzsp} P. Magyar, J. Weynman, A. Zelevinsky, \textit{Symplectic multiple flag varieties
of finite type}, J. Algebra \textbf{230} (2000), no.~1, 245--265.
\bibitem{Popov} V.L. Popov, \textit{Generically multiple transitive algebraic group actions},
Algebraic groups and homogeneous spaces, 481--523, Tata Inst. Fund. Res. Stud. Math., Tata Inst.
Fund. Res., Mumbai, 2007.
\bibitem{Stem} J. Stembridge, \textit{Multiplicity-free products and restrictions of Weyl
characters.}, Representation Theory \textbf{7} (2003), 404--439.
\bibitem{Vin} E.B. Vinberg, \textit{Complexity of actions of reductive groups},
Funkt. Anal. i Pril. \textbf{20} (1986), no.~1, 1--13; English transl.:
Funct. Anal. Appl. \textbf{20} (1986), no~1, 1--11.
\end{thebibliography}
\end{document}